\documentclass[10pt]{amsart}
\title[The Fr\"olicher spectral sequence of certain solvmanifolds]
{The Fr\"olicher spectral sequence of certain solvmanifolds}

\author{Hisashi Kasuya}

\usepackage{amssymb}
\usepackage{amsmath}
\usepackage{amscd}
\usepackage{amstext}
\usepackage{amsfonts}
\usepackage[all]{xy}
\usepackage{multicol}

\theoremstyle{plain}
\newtheorem{theorem}{Theorem}[section] 
\theoremstyle{remark}
\newtheorem{remark}{Remark}
\theoremstyle{lemma}
\newtheorem{lemma}[theorem]{Lemma}
\theoremstyle{assumption}
\newtheorem{assumption}[theorem]{Assumption}
\theoremstyle{condition}

\theoremstyle{definition}
\newtheorem{definition}[theorem]{Definition}
\theoremstyle{proposition}
\newtheorem{proposition}[theorem]{Proposition}
\theoremstyle{corollary}
\newtheorem{corollary}[theorem]{Corollary}
\theoremstyle{remark}
\newtheorem{example}{Example}

\address[H.kasuya]{Graduate school of mathematical science university of tokyo japan }
\curraddr{}
\email{khsc@ms.u-tokyo.ac.jp}


\newcommand{\C}{\mathbb{C}}
\newcommand{\R}{\mathbb{R}}
\newcommand{\N}{\mathbb{N}}

\newcommand{\Z}{\mathbb{Z}}
\newcommand{\g}{\frak{g}}
\newcommand{\n}{\frak{n}}

\begin{document} 

\maketitle
\begin{abstract}
We show that the  Fr\"olicher spectral sequence of a complex parallelizable solvmanifold is degenerate at $E_{2}$-term.
For a   semi-direct product $G=\C^{n}\ltimes_{\phi}N$ of Lie-groups  with  lattice $\Gamma=\Gamma^{\prime}\ltimes \Gamma^{\prime\prime}$ such that $N$ is a nilpotent Lie-group with a left-invariant complex structure and $\phi$ is a semi-simple action, 
we also show that, if the Fr\"olicher spectral sequence of the nilmanifold $N/\Gamma^{\prime\prime}$ is degenerate at $E_{r}$-term for $r\ge 2$, then
the Fr\"olicher spectral  sequence of the solvmanifold $G/\Gamma$ is also degenerate at $E_{r}$-term.
\end{abstract}
\section{Introduction}
Let $M$ be a compact complex manifold.
The Fr\"olicher spectral sequence $E_{\ast}^{\ast,\ast}(M)$ of $M$ is the spectral sequence of the $\C$-valued de Rham complex $(A^{\ast}(M),d)$ which is given by the double complex $(A^{\ast,\ast}(M),\partial,\bar\partial)$.
Consider the differential $d_{\ast}$ on  $E_{\ast}^{\ast,\ast}(M)$.
We denote 
\[r(M)={\rm min}\{r\in \N\vert \forall s\ge r,\, d_{s}=0\}.
\]
It is well-known that $r(M)=1$ if $M$ is K\"ahler.
In general $r(M)=1$ does not hold.
Hence we can say that $r(M)$ measures how far $M$ is from being a K\"ahler manifold.
We are interested in finding non-K\"ahler complex manifolds which has large $r(M)$.

Consider a nilmanifold $G/\Gamma$ with a left-invariant complex structure where $G$ is a simply connected nilpotent Lie-group and $\Gamma$ is a lattice in $G$.
In \cite{CFG}, it is proved that if $G/\Gamma$ is complex parallelizable then $r(G/\Gamma)\le 2$.
However in general for arbitrary $k\in \N$ Rollenske proved that there exists a nilmanifold $G/\Gamma$ with a left-invariant complex structure such that $r(G/\Gamma)\ge k$.

Considering solvmanifolds, it is natural to expect that we can find a wider variety of complex solvmanifolds with large $r(M)$  than the case of nilmanifolds.
In this paper we  consider such expectation.

Our first result extends  results in \cite{CFG}.
\begin{theorem}\label{cos}
Let $G$ be a simply connected complex solvable Lie-group with a lattice $\Gamma$.
Then we have $r(G/\Gamma)\le 2$.
\end{theorem}

We consider a solvable Lie-group $G$ with the following assumption.
\begin{assumption}\label{Ass}
$G$ is the semi-direct product $\C^{n}\ltimes _{\phi}N$ so that:\\
(1) $N$ is a simply connected nilpotent Lie-group with a left-invariant complex structure $J$.\\
Let $\frak a$ and $\n$ be the Lie algebras of $\C^{n}$ and $N$ respectively.\\
(2) For any $t\in \C^{n}$, $\phi(t)$ is a holomorphic automorphism of $(N,J)$.\\
(3) $\phi$ induces a semi-simple action on the Lie-algebra $\n$ of $N$.\\
(4) $G$ has a lattice $\Gamma$. (Then $\Gamma$ can be written by $\Gamma=\Gamma^{\prime}\ltimes_{\phi}\Gamma^{\prime\prime}$ such that $\Gamma^{\prime}$ and $\Gamma^{\prime\prime}$ are  lattices of $\C^{n}$ and $N$ respectively and for any $t\in \Gamma^{\prime}$ the action $\phi(t)$ preserves $\Gamma^{\prime\prime}$.) \\
(5) The inclusion $\bigwedge^{\ast,\ast}\n^{\ast}\subset A^{\ast,\ast}(N/\Gamma^{\prime\prime})$ induces an isomorphism 
\[H^{\ast,\ast}_{\bar\partial}(\n)\cong H^{\ast,\ast}_{\bar\partial }(N/\Gamma^{\prime\prime}).\]
\end{assumption}
Then in \cite{Kd}, we construct an explicit finite dimensional sub-differential bigraded algebra $B^{\ast,\ast} \subset (A^{\ast,\ast}(G/\Gamma),\bar\partial)$ which computes the Dolbeault cohomology $H^{\ast, \ast}_{\bar\partial}(G/\Gamma)$ of $G/\Gamma$.
By this, we can observe that the Dolbeault cohomology $H^{\ast, \ast}_{\bar\partial}(G/\Gamma)$ varies for a choice of a lattice $\Gamma$.
Hence the computation of  $H^{\ast, \ast}_{\bar\partial}(G/\Gamma)$ is more complicated than the computation of $H^{\ast,\ast}_{\bar\partial }(N/\Gamma^{\prime\prime})$.
We prove:

\begin{theorem}\label{sps}
Let $G$ be a Lie group as in Assumption \ref{Ass}.
Then we have:

$\bullet$ If $r(N/\Gamma^{\prime\prime})=1$, then we have $r(G/\Gamma)\le 2$.

$\bullet$ If  $r(N/\Gamma^{\prime\prime})>1$, then we have $r(G/\Gamma)\le r(N/\Gamma^{\prime\prime})$.
\end{theorem}

\begin{corollary}
Let $G$ be a Lie group as in Assumption \ref{Ass}.
Suppose $N$ is complex parallelizable.
Then we have $r(G/\Gamma)\le 2$.
\end{corollary}

\section{Finite-dimensional differential graded algebras of Poincar\'e duality type}
In this section we study the homological algebra of finite-dimensional differential graded algebras  like the theory of harmonic forms on compact Hermitian manifolds.
\begin{definition}
(DGA) A differential graded algebra (DGA) is a graded  commutative $\C$-algebra $A^{\ast}$  with a differential $d$ of degree +1 so that $d\circ d=0$ and $d(\alpha\cdot \beta)=d\alpha\cdot \beta+(-1)^{p}\alpha\cdot d\beta$ for $\alpha\in A^{p}$.

(DBA) A differential bigraded algebra (DBA) is a DGA $(A^{\ast},\bar\partial)$ such that $A^{\ast}$ is bigraded as $A^{r}=\bigoplus _{r=p+q}A^{p,q}$ and the differential $\bar\partial$ has type $(0,1)$.

(BBA) A bidifferential bigraded algebra (BBA) is a DBA $(A^{\ast},\bar\partial)$ with another differential $\partial$ of type $(1,0)$ such that $\partial\bar\partial+\bar\partial\partial=0$.
\end{definition}

Let $A^{\ast}$ be a finite-dimensional graded commutative $\C$-algebra.
\begin{definition}
$A^{\ast}$ is of Poincar\'e duality type (PD-type) if the following conditions hold:

$\bullet$ $A^{\ast<0}=0$ and $A^{0}=\C 1$ where $1$ is the identity element of $A^{\ast}$.

$\bullet$ For some positive integer $n$, $A^{\ast>n}=0$ and $A^{n}=\C v$ for $v\not=0$.

$\bullet$ For any $0<i<n$ the bi-linear map $A^{i}\times A^{n-i}\ni (\alpha,\beta)\mapsto \alpha\cdot \beta\in A^{n}$ is non-degenerate.
\end{definition}

Suppose $A^{\ast}$ is of PD-type.
Let $h$ be a Hermitian metric on $A^{\ast}$ which is compatible with the grading.
Take $v\in A^{n}$ such that $h(v,v)=1$.
Define the $\C$-anti-linear map $\bar\ast: A^{i}\to A^{n-i}$ as $\alpha\cdot \bar\ast\beta=h(\alpha,\beta)v$.

\begin{definition}
A finite-dimensional DGA $(A^{\ast},d)$ is of PD-type if  the following conditions hold:

$\bullet$  $A^{\ast}$ is  a finite-dimensional graded $\C$-algebra of PD-type.

$\bullet$ $dA^{n-1}=0$ and $dA^{0}=0$.
\end{definition}

Let $(A^{\ast},d)$ be a finite-dimensional DGA of PD-type.
Denote $d^{\ast}=-\bar\ast d\bar\ast$.
\begin{lemma}\label{add}
We have $h(d\alpha, \beta)=h(\alpha,d^{\ast}\beta)$ for $\alpha\in A^{i-1}$ and $\beta\in A^{i}$.
\end{lemma}
\begin{proof}
By $dA^{n-1}=0$, we have
\[d\alpha \cdot\bar\ast \beta=d(\alpha\cdot\bar\ast\beta)-(-1)^{i-1}\alpha\cdot (d\bar\ast\beta)=(-1)^{p}\alpha\cdot (d\bar\ast\beta)=\alpha\cdot (\bar\ast\bar\ast d\bar\ast \beta).\]
Hence we have
\[h(d\alpha,\beta)v=h(\alpha,d^{\ast}\beta)v.\]
\end{proof}

Define $\Delta=dd^{\ast}+d^{\ast}d$.
and ${\mathcal H}^{\ast}(A)=\ker \Delta$.
By Lemma \ref{add} and finiteness of the dimension of $A^{\ast}$, we can easily show the following lemma like the proof of \cite[Theorem 5.23]{Vo}.

\begin{lemma}\label{fiho}
We have the decomposition
\[A^{p}={\mathcal H}^{p}(A)\oplus \Delta(A^{p})={\mathcal H}^{p}(A)\oplus d(A^{p-1})\oplus d^{\ast}(A^{p+1}).
\]
By this decomposition, the inclusion ${\mathcal H}^{\ast}(A)\subset A^{\ast}$ induces a  isomorphism
\[{\mathcal H}^{p}(A)\cong H^{p}(A)
\]
of vector spaces.
\end{lemma}

\begin{lemma}\label{rm}
Let $(A^{\ast},d)$ be a finite-dimensional DGA of PD-type.
Then the cohomology algebra $H^{\ast}(A)$ is  a finite-dimensional graded commutative $\C$-algebra of PD-type.
\end{lemma}
\begin{proof}
Since the restriction $\bar\ast: {\mathcal H}^{i}(A)\to {\mathcal H}^{n-i}(A)$ is also an isomorphism,
the linear map $ {\mathcal H}^{i}(A)\times  {\mathcal H}^{n-i}(A)\ni(\alpha,\beta)\to \alpha\cdot\beta \in{\mathcal H}^{n}(A)=A^{n}$ is non-degenerate.
Hence the lemma follows from Lemma \ref{fiho}.
\end{proof}
\begin{lemma}\label{injPD}
Let $(A^{\ast},d)$ be a finite-dimensional DGA of PD-type and $B^{\ast}\subset A^{\ast}$ be a sub-DGA  of PD-type.
Then the inclusion $B^{\ast}\subset A^{\ast}$ induces an injection
\[H^{\ast}(B^{\ast})\hookrightarrow H^{\ast}( A^{\ast}).
\]
\end{lemma}
\begin{proof}
By the inclusion ${\mathcal H}^{\ast}(B)\subset {\mathcal H}^{\ast}(A)$ the Lemma follows from Lemma \ref{fiho}.

\end{proof}

\begin{proposition}\label{fispe}
Let $(A^{\ast,\ast}, d=\partial+\bar\partial)$ be a  finite-dimensional BBA  such that $({\rm Tot} A^{\ast,\ast}, \partial+\bar\partial) $ is  a  finite-dimensional DGA of PD-type.
Let $B^{\ast,\ast}\subset A^{\ast,\ast}$ be a sub-BBA such that $({\rm Tot} B^{\ast,\ast}, \partial+\bar\partial)$ is  a  finite-dimensional DGA of PD-type.
Consider the spectral sequences  $E^{\ast,\ast}_{\ast}(A^{\ast,\ast})$ and $E^{\ast,\ast}_{\ast}(B^{\ast,\ast})$ given by the BBA-structure.
Then for each $r$, ${\rm Tot}E^{\ast,\ast}_{r}(A)$ and ${\rm Tot}E^{\ast,\ast}_{r}(B)$ are finite-dimensional DGAs of PD-type and
the inclusion $B^{\ast,\ast}\subset A^{\ast,\ast}$ induces an injection $E^{\ast,\ast}_{r}(A^{\ast,\ast})\hookrightarrow E^{\ast,\ast}_{r}(B^{\ast,\ast})$.
\end{proposition}
\begin{proof}
We will prove the proposition inductively.
By the assumption ${\rm Tot}E^{\ast,\ast}_{0}(A^{\ast,\ast})\cong  ({\rm Tot} A^{\ast,\ast}, \bar \partial)$ and  ${\rm Tot}E^{\ast,\ast}_{0}(B^{\ast,\ast})\cong  ({\rm Tot} B^{\ast,\ast}, \bar\partial)$
are finite-dimensional DGAs of PD-type.
Suppose that for some $r$  ${\rm Tot}E^{\ast,\ast}_{r}(A)$ and ${\rm Tot}E^{\ast,\ast}_{r}(B)$ are finite-dimensional DGAs of PD-type and
the inclusion $B^{\ast,\ast}\subset A^{\ast,\ast}$ induces an injection $E^{\ast,\ast}_{r}(A^{\ast,\ast})\subset E^{\ast,\ast}_{r}(B^{\ast,\ast})$.
Since we have $H^{\ast}({\rm Tot}E^{\ast,\ast}_{r}(A))\cong {\rm Tot}E^{\ast,\ast}_{r+1}(A)$ and $H^{\ast}({\rm Tot}E^{\ast,\ast}_{r}(B))\cong {\rm Tot}E^{\ast,\ast}_{r+1}(B)$,   by Lemma \ref{rm} and \ref{injPD}, ${\rm Tot}E^{\ast,\ast}_{r+1}(A)$ and ${\rm Tot}E^{\ast,\ast}_{r+1}(B)$ are finite-dimensional DGAs of PD-type and the induced map
 $E^{\ast,\ast}_{r+1}(A^{\ast,\ast})\to E^{\ast,\ast}_{r+1}(B^{\ast,\ast})$ is injective.

\end{proof}

\section{Proof of Theorem \ref{cos}}
Let $G$ be a simply connected  solvable Lie-group.
Denote by $ \g_{+}$ (resp. $ \g_{-}$) the Lie algebra of the left-invariant holomorphic (anti-holomorphic) vector fields on $G$.
As a real Lie algebra we have an isomorphism  $ \g_{+}\cong \g_{-}$ by  complex  conjugation.
Let $N$ be the nilradical  of $G$.
We can take a  simply connected complex nilpotent subgroup $C\subset G$  such that $G=C\cdot N$ (see \cite{dek}).
Since $C$ is nilpotent, the map
\[C\ni c \mapsto ({\rm Ad}_{c})_{s}\in {\rm Aut}(\g_{+})\]
is a homomorphism where $({\rm Ad}_{c})_{s}$ is the semi-simple part of ${\rm Ad}_{c}$.
Denote by $\bigwedge \g_{+}^{\ast}$ (resp. $\bigwedge \g_{-}^{\ast}$) the sub-DGA of $(A^{\ast,0}(G/\Gamma),\partial)$ (resp. $(A^{0,\ast}(G/\Gamma),\bar\partial)$ which consists of the left-invariant  holomorphic (anti-holomorphic) forms.
As a DGA, we have an isomorphism between $(\bigwedge \g_{+}^{\ast},\partial)$ and $(\bigwedge \g_{-}^{\ast},\bar\partial)$ given by  complex conjugation.

We have a basis $X_{1},\dots,X_{n}$ of $\g_{+}$ such that $({\rm Ad}_{c})_{s}={\rm diag} (\alpha_{1}(c),\dots,\alpha_{n}(c))$ for $c\in C$ where $\alpha_{1},\dots, \alpha_{n}$ are holomorphic characters.
Let $x_{1},\dots, x_{n}$ be the basis of $\g^{\ast}_{+}$ which is dual to $X_{1},\dots ,X_{n}$.
For a multi-index $I=\{i_{1},\dots ,i_{p}\}$ we write $x_{I}=x_{i_{1}}\wedge\dots \wedge x_{i_{p}}$,  and $\alpha_{I}=\alpha_{i_{1}}\cdots \alpha_{i_{p}}$.
\begin{theorem}{\rm (\cite[Corollary 6.2 and its proof]{KDD})}\label{MMTT}
Let $B^{\ast}_{\Gamma}$ be the subcomplex of $(A^{0,\ast}(G/\Gamma),\bar\partial) $ defined as
\[B^{\ast}_{\Gamma}=\left\langle \frac{\bar\alpha_{I}}{\alpha_{I} }\bar x_{I}{\Big \vert}\left(\frac{\bar\alpha_{I}}{\alpha_{I}}\right)_{ \vert_{\Gamma}}=1\right\rangle.
\]
Then the inclusion $B^{\ast}_{\Gamma}\subset A^{0,\ast}(G/\Gamma) $ induces an isomorphism
\[H^{\ast}(B^{\ast}_{\Gamma})\cong H^{0,\ast}(G/\Gamma).
\]
\end{theorem}
By this theorem, since $G/\Gamma$ is complex parallelizable, for the DBA $(\bigwedge \g_{+}^{\ast}\otimes B^{\ast}_{\Gamma},\bar\partial)$, the inclusion $\bigwedge \g^{\ast}\otimes B^{\ast}_{\Gamma}\subset A^{\ast,\ast}$ induce an isomorphism
\[\bigwedge \g^{\ast}_{+}\otimes H^{\ast}(B^{\ast}_{\Gamma})\cong H^{\ast,\ast}(G/\Gamma).\]
For a holomorphic character $\nu$, we define the subspaces
\[V_{\nu}=\left\langle x_{I}{\Big \vert} \alpha_{I}=\frac{1}{\nu}\right\rangle\]
and
\[V_{\bar \nu}=\left\langle\bar x_{I}{\Big \vert} \alpha_{I}=\frac{1}{\nu}\right\rangle\]
of $\bigwedge \g^{\ast}_{+}$ and $\bigwedge \g^{\ast}_{-}$ respectively.
We consider the weight decomposition
\[\bigwedge \g^{\ast}_{+}=\bigoplus V_{\nu_{k}}\]
and
\[\bigwedge \g^{\ast}_{-}=\bigoplus V_{\bar\nu_{k}}.\]
Then we have 
\[B^{\ast}_{\Gamma}=\bigoplus_{\left( \frac{\nu_{k}}{\bar\nu_{k} }\right)_{ \vert_{\Gamma}}=1}  \frac{\nu_{k}}{\bar\nu_{k} }V_{\bar\nu_{k}} \]
and hence 
\[\bigwedge \g^{\ast}_{+}\otimes B^{\ast}_{\Gamma}=\bigoplus_{\left( \frac{\nu_{k}}{\bar\nu_{k} }\right)_{ \vert_{\Gamma}}=1} \left( \nu_{k}\bigwedge \g_{+}^{\ast} \right)\otimes \left(\frac{1}{\bar\nu_{k} }V_{\bar\nu_{k}} \right).\]
Regard $\bigwedge \g^{\ast}_{+}\otimes B^{\ast}_{\Gamma}$ as a BBA $(\bigwedge \g^{\ast}_{+}\otimes B^{\ast}_{\Gamma},\partial,\bar\partial)$.
Since each $\nu_{k}$ is holomorphic, we have $\bar\partial\left( \nu_{k}\bigwedge \g_{+}^{\ast} \right)=0$ and $\partial \left(\frac{1}{\bar\nu_{k} }V_{\bar\nu_{k}} \right)=0$.
Hence for each $\nu_{k}$ the double complex  $\left( \nu_{k}\bigwedge \g_{+}^{\ast} \right)\otimes \left(\frac{1}{\bar\nu_{k} }V_{\bar\nu_{k}} \right)$ is the tensor product of $\left(\nu_{k}\bigwedge \g^{\ast}_{+}, \partial\right)$ and $\left(\frac{1}{\bar\nu_{k} }V_{\bar\nu_{k}},\bar\partial\right)$.
We consider the spectral sequence $E_{\ast}^{\ast,\ast}\left(\bigwedge \g^{\ast}_{+}\otimes B^{\ast}_{\Gamma}\right)$ of the double complex $\bigwedge \g^{\ast}_{+}\otimes B^{\ast}_{\Gamma}$.
Then we have
\[E_{2}^{\ast,\ast}\left(\bigwedge \g^{\ast}_{+}\otimes B^{\ast}_{\Gamma}\right)=\bigoplus_{\left( \frac{\nu_{k}}{\bar\nu_{k} }\right)_{ \vert_{\Gamma}}=1}  H^{\ast}_{\partial} \left( \nu_{k}\bigwedge \g_{+}^{\ast} \right)\otimes H^{\ast}_{\bar\partial}\left(\frac{1}{\bar\nu_{k} }V_{\bar\nu_{k}} \right)
\]
By the K\"unneth theorem (see \cite{ma}), we have
\[H^{\ast}_{d}\left(\bigwedge \g^{\ast}_{+}\otimes B^{\ast}_{\Gamma}\right)=\bigoplus_{\left( \frac{\nu_{k}}{\bar\nu_{k} }\right)_{ \vert_{\Gamma}}=1} {\rm Tot}\left( H^{\ast}_{\partial} \left( \nu_{k}\bigwedge \g_{+}^{\ast} \right)\otimes H^{\ast}_{\bar\partial}\left(\frac{1}{\bar\nu_{k} }V_{\bar\nu_{k}} \right)\right).
\]
Hence the spectral  sequence $E_{\ast}^{\ast,\ast}\left(\bigwedge \g^{\ast}_{+}\otimes B^{\ast}_{\Gamma}\right)$ is degenerate at $E_{2}$-term.
By Theorem \ref{MMTT} and \cite[Theorem 3.5]{mc}), for the Fr\"olicher spectral sequence $E_{\ast}^{\ast,\ast}(G/\Gamma)$
we have 
$E_{2}^{\ast,\ast}\left(\bigwedge \g^{\ast}_{+}\otimes B^{\ast}_{\Gamma}\right)\cong E_{2}^{\ast,\ast}(G/\Gamma)$.
Hence Theorem  \ref{cos} follows.

\section{Proof of Theorem \ref{sps}}\label{th2}
Let $G$ be a Lie-group as in Assumption \ref{Ass}.
Consider the decomposition $\n_{\C}=\n^{1,0}\oplus \n^{0,1}$.
By the condition (2), this decomposition is a direct sum of $\C^{n}$-modules.
By the condition (3) we have a basis $Y_{1},\dots ,Y_{m}$ of $\n^{1,0}$ such that the action of $\C^{n}$ on $\n^{1,0}$ is represented by
$\phi(t)={\rm diag} (\alpha_{1}(t),\dots, \alpha_{m} (t))$ where $\alpha_{1}(t),\dots, \alpha_{m} (t)$ are characters of $\C^{n}$.
Since $Y_{j}$ is a left-invariant vector field on $N$,
for $(t,x)\in \C^{n}\ltimes _{\phi} N$, we have
\[\alpha_{j}(t)(Y_{j})_{x}=L_{(t,x)} \left(\alpha_{j}(0)(Y_{j})_{e}\right)
\] where $L_{(t,x)} $ is the left-action
and hence
the vector field $\alpha_{j}Y_{j}$ on $\C^{n}\ltimes _{\phi} N$ is  left-invariant.
Hence we have a basis $X_{1},\dots,X_{n}, \alpha_{1}Y_{1},\dots ,\alpha_{m}Y_{m}$ of $\g^{1,0}$.
Let $x_{1},\dots,x_{n}, \alpha^{-1}_{1}y_{1},\dots ,\alpha_{m}^{-1}y_{m}$ be the  basis of $\bigwedge^{1,0}\g^{\ast}$ which is dual to $X_{1},\dots,X_{n}, \alpha_{1}Y_{1},\dots ,\alpha_{m}Y_{m}$.
Then we have 
\[\bigwedge ^{p,q}\g^{\ast}=\bigwedge ^{p}\langle x_{1},\dots ,x_{n}, \alpha^{-1}_{1}y_{1},\dots ,\alpha^{-1}_{m}y_{m}\rangle\otimes \bigwedge ^{q}\langle \bar x_{1},\dots ,\bar x_{n}, \bar\alpha^{-1}_{1}\bar y_{1},\dots ,\bar\alpha^{-1}_{m}\bar y_{m}\rangle.
\]


\begin{remark}\label{nodd}
We also consider the direct sum $\C^{n}\oplus \n$.
Then we have
\[\bigwedge ^{p,q} (\C^{n}\oplus \n)_{\C}^{\ast}=\bigwedge ^{p}\langle x_{1},\dots ,x_{n}, y_{1},\dots ,y_{m}\rangle\otimes \bigwedge ^{q}\langle \bar x_{1},\dots ,\bar x_{n}, \bar y_{1},\dots ,\bar y_{m}\rangle.
\]
The space $A^{\ast,\ast}(G/\Gamma)$ of differential forms on $G/\Gamma$ is identified with the space of $\Gamma$-invariant differential forms on the Lie-group $\C^{n}\ltimes _{\phi} N$.

All the spaces of forms which we consider in this section can be regarded as  subspaces in the space of differential forms on the complex manifold $\C^{n}\times  N$.
In particular we do not need to specify the differentials in the double complexes.

\end{remark}

By \cite[Lemma 2.2]{Kd}, 
 for each $i$ we can take a unique pair of unitary characters $\beta_{i}$ and $\gamma_{i}$ on $\C^{n}$ such that $\alpha_{i}\beta_{i}^{-1}$ and $\bar\alpha\gamma^{-1}_{i}$ are holomorphic.

\begin{theorem}{\rm (\cite[Corollary 4.2]{Kd})}\label{CORR}
Let $G, \Gamma$ be as above.
We 
define   the sub-DBA $B^{\ast,\ast}$ of $A^{\ast,\ast}(G/\Gamma)$ given by
\[B^{p,q}=\left\langle x_{I}\wedge \alpha^{-1}_{J}\beta_{J}y_{J}\wedge \bar x_{K}\wedge \bar \alpha^{-1}_{L}\gamma_{L}\bar y_{L}{\Big \vert} \begin{array}{cc}\vert I\vert+\vert K\vert=p,\, \vert J\vert+\vert L\vert=q \\  {\rm the \, \,  restriction \, \,  of }\, \, \beta_{J}\gamma_{L}\, \,  {\rm on  \, \, \Gamma \, \, is\, \,  trivial}\end{array}\right\rangle,
\]
 the inclusion $B^{\ast,\ast}\subset A^{\ast,\ast}(G/\Gamma)$ induces a cohomology isomorphism
\[H^{\ast,\ast}_{\bar \partial}(B^{\ast,\ast})\cong H^{\ast,\ast}_{\bar \partial}(G/\Gamma).
\]
\end{theorem}
\begin{remark}
For $B^{p,q}\ni x_{I}\wedge \alpha^{-1}_{J}\beta_{J}y_{J}\wedge \bar x_{K}\wedge \bar \alpha^{-1}_{L}\gamma_{L}\bar y_{L}$, since $\alpha^{-1}_{J}\beta_{J}$ and $\bar \alpha^{-1}_{L}\gamma_{L}$ are holomorphic,    we have
\begin{multline*}
\bar\partial\left(x_{I}\wedge \alpha^{-1}_{J}\beta_{J}y_{J}\wedge \bar x_{K}\wedge \bar \alpha^{-1}_{L}\gamma_{L}\bar y_{L}\right)\\
=(-1)^{\vert I\vert}x_{I}\wedge \alpha^{-1}_{J}\beta_{J}\bar\partial(y_{J})\wedge \bar x_{K}\wedge \bar \alpha^{-1}_{L}\gamma_{L}\bar y_{L}+(-1)^{\vert I\vert+\vert J\vert+\vert K\vert}x_{I}\wedge \alpha^{-1}_{J}\beta_{J}\bar y_{J}\wedge \bar x_{K}\wedge \bar \alpha^{-1}_{L}\gamma_{L}\bar\partial(\bar y_{L}).
\end{multline*}
Hence by the injection
\[B^{p,q}\ni x_{I}\wedge \alpha^{-1}_{J}\beta_{J}y_{J}\wedge \bar x_{K}\wedge \bar \alpha^{-1}_{L}\gamma_{L}\bar y_{L}\mapsto x_{I}\wedge y_{J}\wedge \bar x_{K}\wedge \bar y_{L} \in \bigwedge ^{p,q} (\C^{n}\oplus \n)_{\C}^{\ast},
\]
 $(B^{\ast,\ast},\bar\partial)$ can be regarded as a sub-DBA of the DBA $(\bigwedge ^{\ast,\ast} (\C^{n}\oplus \n)_{\C}^{\ast},\bar\partial)$.
However, considering the differential $\partial$, $(B^{\ast,\ast}, \partial, \bar\partial)$  is not a sub-BBA of the BBA $(\bigwedge ^{\ast,\ast} (\C^{n}\oplus \n)_{\C}^{\ast}, \partial, \bar\partial)$
\end{remark}

The action $\phi$ induces the action of $\C^{n}$ on the BBA $\left(\bigwedge ^{\ast,\ast} \n^{\ast}_{\C}, \partial,\bar\partial\right)$.
We consider the weight decomposition
\[\bigwedge ^{\ast,\ast} \n^{\ast}_{\C}=\bigoplus_{\mu_{i}}A_{\mu_{i}}\]
of this action.
Then each $A_{\mu_{i}}$ is a sub-Double-complex of $\bigwedge ^{\ast,\ast} \n^{\ast}_{\C}$.
Let $B^{\ast,\ast}$ be the DBA as in Theorem \ref{CORR}.
Then we have
\[B^{\ast,\ast}=\bigoplus_{\mu_{i}, (\lambda_{i})_{\vert_{\Gamma}}=1}\bigwedge \C^{n}\otimes \left(\lambda_{i}\mu_{i}^{-1}A_{\mu_{i}}\right)\]
where $\lambda_{i} $ are unique unitary characters such that $\lambda_{i}\mu_{i}^{-1}$ are holomorphic.
If $(\mu_{i})_{\vert_{\Gamma}}=1$, then $\mu_{i}$ is unitary and so $\lambda_{i}=\mu_{i}$.
If $(\mu_{i})_{\vert_{\Gamma}}\not=1$ and $(\lambda_{i})_{\vert_{\Gamma}}=1$, then $\lambda_{i}\mu_{i}^{-1}$ is a non-trivial holomorphic character.
Consider the double-complexes
\[C^{\ast,\ast}=\bigoplus_{(\mu_{i})_{\vert_{\Gamma}}=1}\bigwedge \C^{n}\otimes A_{\mu_{i}} \]
and
\[D^{\ast,\ast}=\bigoplus_{(\mu_{i})_{\vert_{\Gamma}}\not=1,\, (\lambda_{i})_{\vert_{\Gamma}}=1}\left(\lambda_{i}\mu_{i}^{-1}\bigwedge \C^{n}\right)\otimes A_{\mu_{i}}.\]
Then we have
\[B^{\ast,\ast} =C^{\ast,\ast}\oplus D^{\ast,\ast}.\]
Consider the spectral sequences $E^{\ast,\ast}_{\ast}(B)$, $E^{\ast,\ast}_{\ast}(C)$ and $E^{\ast,\ast}_{\ast}(D)$ of the double complexes  $B^{\ast,\ast}$, $C^{\ast,\ast}$ and $D^{\ast,\ast}$ respectively.
Then we have:

\begin{lemma}\label{e2e}
\[E_{2}^{\ast,\ast}(B)=E_{2}^{\ast,\ast}(C).\]
\end{lemma}
\begin{proof} 
We have
\[E_{2}^{\ast,\ast}(D)=\bigoplus_{(\mu_{i})_{\vert_{\Gamma}}\not=1,\, (\lambda_{i})_{\vert_{\Gamma}}=1}H^{\ast,\ast}_{\partial}\left(\left(\lambda_{i}\mu_{i}^{-1}\bigwedge \C^{n}\right)\otimes H^{\ast,\ast}_{\bar\partial}\left(A_{\mu_{i}}\right)\right).\]
By the K\"unneth theorem (see \cite{ma}), we have
\[
H^{\ast,\ast}_{\partial}\left(\left(\lambda_{i}\mu_{i}^{-1}\bigwedge \C^{n}\right)\otimes H^{\ast,\ast}_{\bar\partial}\left(A_{\mu_{i}}\right)\right)=H^{\ast,\ast}_{\partial}\left(\lambda_{i}\mu_{i}^{-1}\bigwedge \C^{n}\right)\otimes H^{\ast,\ast}_{\partial}\left(H^{\ast,\ast}_{\bar\partial}\left(A_{\mu_{i}}\right)\right).\]
If $\lambda_{i}\mu_{i}^{-1}$ is a non-trivial holomorphic character,  then the cohomology $H^{\ast,\ast}_{\partial}\left(\lambda_{i}\mu_{i}^{-1}\bigwedge \C^{n}\right)$ is identified with the Lie algebra cohomology   of the Abelian Lie algebra $\C^{n}$ with values in a non-trivial $1$-dimensional representation and hence we have $H^{\ast,\ast}_{\partial}\left(\lambda_{i}\mu_{i}^{-1}\bigwedge \C^{n}\right)=0$.
This implies 
\[E_{2}^{\ast,\ast}(D)=\bigoplus_{(\mu_{i})_{\vert_{\Gamma}}\not=1,\, (\lambda_{i})_{\vert_{\Gamma}}=1}H^{\ast,\ast}_{\partial}\left(\lambda_{i}\mu_{i}^{-1}\bigwedge \C^{n}\right)\otimes H^{\ast,\ast}_{\partial}\left(H^{\ast,\ast}_{\bar\partial}\left(A_{\mu_{i}}\right)\right)=0.\]
Hence the lemma follows.
\end{proof}


Since we have 
\[C^{p,q}=\left\langle x_{I}\wedge y_{J}\wedge \bar x_{K}\wedge \bar y_{L}{\Big \vert} \begin{array}{cc}\vert I\vert+\vert K\vert=p,\, \vert J\vert+\vert L\vert=q \\  {\rm the \, \,  restriction \, \,  of }\, \, \alpha^{-1}_{J} \bar \alpha^{-1}_{L}\, \,  {\rm on  \, \, \Gamma \, \, is\, \,  trivial}\end{array}\right\rangle,
\]
we have  $C^{\ast,\ast}\subset \bigwedge (\C^{n}\oplus \n)_{\C}^{\ast}$ by Remark \ref{nodd}, 
$C^{\ast,\ast}$ is closed under wedge product and we have $\bar\ast(C^{\ast,\ast})\subset C^{\ast,\ast}$ where $\bar\ast$ is the Hodge star operator of the left-invariant Hermitian metric
\[x_{1}\bar x_{1}+\dots +x_{n}\bar x_{n}+y_{1}\bar y_{1}+\dots +y_{m}\bar y_{m}\]
on $\C^{n}\times N$.
Hence ${\rm Tot}C^{\ast,\ast}$ is a finite dimensional DGA of PD-type.
By Proposition \ref{fispe}, we have an inclusion $E^{\ast,\ast}_{r}(C)\hookrightarrow E^{\ast,\ast}_{r}(\bigwedge (\C^{n}\oplus \n)_{\C}^{\ast})$.
Hence if the spectral sequence $E^{\ast,\ast}_{\ast}(\bigwedge (\C^{n}\oplus \n)_{\C}^{\ast})$ is degenerate at $E_{r}$-term, then  the spectral sequence $E^{\ast,\ast}_{\ast}(C)$ is also degenerate at $E_{r}$-term.
By the condition (5) in Assumption \ref{Ass}, the spectral sequence $E^{\ast,\ast}_{\ast}(\bigwedge (\C^{n}\oplus \n)_{\C}^{\ast})$ is degenerate at $E_{r}$-term if and only if the Fr\"olicher spectral sequence $E_{\ast}^{\ast,\ast}(N/\Gamma^{\prime\prime}) $ of the nilmanifold $N/\Gamma^{\prime\prime}$ is  degenerate at $E_{r}$-term.
By Theorem \ref{CORR} and Lemma \ref{e2e} and \cite[Theorem 3.5]{mc}, we have
\[E^{\ast,\ast}_{r}(G/\Gamma)\cong E^{\ast,\ast}_{r}(B)=E^{\ast,\ast}_{r}(C)\]
for $r\ge2$.
Hence Theorem \ref{sps} follows.

\section{Examples and  remarks}
In this section we show that Theorem \ref{sps} is sharp, that is:
\begin{remark}\label{rre}
We can construct pairs $(G,\Gamma)$ as in Assumption \ref{Ass} such that:

$\bullet$ 
 $r(G/\Gamma)=2$ but $r(N/\Gamma^{\prime\prime})=1$

$\bullet$ 
 $r(G/\Gamma)< r(N/\Gamma^{\prime\prime})$.
\end{remark}

\begin{example}
Let $G=\C\ltimes _{\phi}\C^{2}$ such that $\phi(x+\sqrt{-1}y)=\left(
\begin{array}{cc}
e^{x}& 0  \\
0&    e^{-x}  
\end{array}
\right)$.
Then for some $a\in \R$  the matrix $\left(
\begin{array}{cc}
e^{a}& 0  \\
0&    e^{-a}  
\end{array}
\right)$
 is conjugate to an element of $SL(2,\Z)$.
 Hence for any $0\not=b\in \R$ we have a lattice $\Gamma=(a\Z+b\sqrt{-1}\Z )\ltimes \Gamma^{\prime\prime}$ such that $\Gamma^{\prime\prime} $ is a lattice of $\C^{2}$.

If $b=2\pi$, then the Dolbeault cohomology $H^{\ast,\ast}(G/\Gamma)$ is isomorphic to the Dolbeault cohomology of the complex $3$-torus (see \cite{Kd}).
In this case we have 
\[\dim H^{\ast}(G/\Gamma)=2<6=\dim H^{1,0}_{\bar\partial}(G/\Gamma)+\dim H^{1,0}_{\bar\partial}(G/\Gamma).\]
This implies $r(G/\Gamma)>1$.
Hence the first assertion of Remark \ref{rre} follows.
\begin{remark}
In \cite{KH}, in case $N$ is Abelian we give a condition for $r(G/\Gamma)=1$.
\end{remark}
\end{example}

\begin{example}
Let $G=\C\ltimes_{\phi} N$ such that 
\[N=\left\{\left(
\begin{array}{cccc}
1&  \bar z&  \frac{1}{2}\bar z^{2}&w  \\
0&     1&\bar z& v   \\
0& 0&1 &z \\
0&0&0&1
\end{array}
\right): z,v,w\in \C\right\}\]
and $\phi$ is given by
\[\phi(s+\sqrt{-1}t)\left(
\begin{array}{cccc}
1&  \bar z&  \frac{1}{2}\bar z^{2}&w  \\
0&     1&\bar z& v   \\
0& 0&1 &z \\
0&0&0&1
\end{array}
\right)=
\left(
\begin{array}{cccc}
1& e^{-\pi \sqrt{-1} t}\bar z&  \frac{1}{2} e^{-2\pi \sqrt{-1} t}\bar z^{2}&e^{-\pi  \sqrt{-1} t} w  \\
0&     1&e^{-\pi  \sqrt{-1}t} \bar z& v   \\
0& 0&1 &e^{\pi  \sqrt{-1}t}z \\
0&0&0&1
\end{array}
\right).
\]
We have a lattice 
\[\Gamma=(\Z+\sqrt{-1}\Z)\ltimes \Gamma^{\prime\prime}
\]
such that
\[\Gamma^{\prime\prime}=N\cap GL(4, \Z[\sqrt{-1}]).
\]
Then $G$ is a Lie-group as Assumption \ref{Ass} and we have
$x_{1}=ds+\sqrt{-1}t$, $y_{1}=dz$, $y_{2}=dv-\bar z dz$ and $y_{3}=dw+\bar z dv -\frac{1}{2}\bar z^{2}dz$ where $x_{i}, y_{i}$ are as Section \ref{th2}.
We have $dy_{1}=0$, $dy_{2}=y_{1}\wedge \bar y_{1}$, and $dy_{3}=\bar y_{1}\wedge y_{2}$.
It is known that $r(N/\Gamma^{\prime\prime})\ge 3$ (see \cite[Example 1]{CFGU}).
The second assertion of Remark \ref{rre} follows from the following proposition.
\begin{proposition}
We have $r(G/\Gamma)=1$.
\end{proposition}

\begin{proof}
Since $G/\Gamma$ is real parallelizable,
the Euler class and all  Chern classes are trivial.
Hence we have
\[\sum^{8}_{k=0}(-1)^{k}\dim H^{\ast} (G/\Gamma)=\chi(G/\Gamma)=0\]
and
by the Hirzebruch--Riemann--Roch theorem (see \cite{Hir}), for each $i\in \{1,2,3,4\}$, we have
\[\sum (-1)^{q}\dim H^{i,q}_{\bar \partial}(G/\Gamma)=\chi(G/\Gamma, \bigwedge^{i}T^{\ast}G/\Gamma)=0
\]
and so
\[\sum_{(p,q)=(0,0)}^{(4,4)}(-1)^{p,q} \dim H^{p,q}_{\bar \partial}(G/\Gamma)=0.
\]
Hence by the Poincar\'e Duality and Serre Duality, it is sufficient to show the equalities
\[\dim H^{1,0}_{\bar\partial}(G/\Gamma)+\dim H^{0,1}_{\bar\partial}(G/\Gamma)=\dim H^{1}(G/\Gamma),
\]
\[\dim H^{2,0}_{\bar\partial}(G/\Gamma)+\dim H^{1,1}_{\bar\partial}(G/\Gamma)+\dim H^{0,2}_{\bar\partial}(G/\Gamma)=\dim H^{2}(G/\Gamma),\]
\[\dim H^{3,0}_{\bar\partial}(G/\Gamma)+\dim H^{2,1}_{\bar\partial}(G/\Gamma)+\dim H^{1,2}_{\bar\partial}(G/\Gamma)+\dim H^{0,3}_{\bar\partial}(G/\Gamma)=\dim H^{3}(G/\Gamma).
\]

We consider $B^{\ast,\ast}$ as Theorem \ref{CORR}.
We have an isomorphism $H^{\ast,\ast}_{\bar\partial}(G/\Gamma)\cong H^{\ast,\ast}_{\bar\partial}(B^{\ast,\ast})$.
We also have  $H^{\ast}(G/\Gamma)\cong H^{\ast}({\rm Tot}B^{\ast,\ast})$.
Then we have
\[
B^{1,0}=\langle x_{1}, \, y_{2}\rangle,\,\, B^{0,1}=\langle \bar x_{1},\,  \bar y_{2}\rangle,
\]
\[B^{2,0}=\langle x_{1}\wedge y_{2}, \, y_{1}\wedge y_{3}\rangle, \, B^{0,2}= \langle \bar x_{1}\wedge\bar y_{2}, \, \bar y_{1}\wedge \bar y_{3}\rangle,
\]
\[
B^{1,1}=\langle x_{1}\wedge \bar x_{1},\, x_{1}\wedge \bar y_{2},\, y_{1}\wedge \bar y_{1},\, y_{1}\wedge \bar y_{3},\, y_{2}\wedge \bar x_{1},\, y_{2}\wedge \bar y_{2},\, y_{3}\wedge \bar y_{1},\, y_{3}\wedge \bar y_{3}\rangle,
\]
\[B^{3,0}=\langle x_{1}\wedge y_{1}\wedge y_{3},\, y_{1}\wedge y_{2}\wedge y_{3}\rangle, \, B^{0,3}=\langle \bar x_{1}\wedge\bar y_{1}\wedge\bar y_{3},\,\bar y_{1}\wedge \bar y_{2}\wedge\bar y_{3}\rangle,
\]
\begin{multline*}
B^{2,1}=\langle x_{1}\wedge y_{1}\wedge \bar y_{1},\, x_{1}\wedge y_{1}\wedge \bar y_{3},\, x_{1}\wedge y_{2}\wedge \bar x_{1}, \, x_{1}\wedge y_{2}\wedge\bar y_{2},\, x_{1}\wedge y_{3}\wedge \bar y_{1},\, x_{1}\wedge y_{3}\wedge \bar y_{3},\, y_{1}\wedge y_{2}\wedge \bar y_{1},\, y_{1}\wedge y_{2}\wedge \bar y_{3} \\
y_{1}\wedge y_{3}\wedge \bar x_{1},\, y_{1}\wedge y_{3}\wedge \bar y_{2},\, y_{2}\wedge y_{3}\wedge \bar y_{1},\ y_{2}\wedge y_{3}\wedge \bar y_{3} \rangle,
\end{multline*}
\begin{multline*}
B^{1,2}=\langle x_{1}\wedge \bar x_{1}\wedge \bar y_{2},\, x_{1}\wedge \bar y_{1}\wedge \bar y_{3},\, y_{1}\wedge \bar x_{1}\wedge \bar y_{1},\, y_{1}\wedge \bar x_{1}\wedge \bar y_{3},\, y_{1}\wedge  \bar y_{1}\wedge \bar y_{2}, \, y_{1}\wedge \bar y_{2}\wedge \bar y_{3}\\
y_{2}\wedge \bar x_{1}\wedge \bar y_{2},\, y_{2}\wedge \bar y_{1}\wedge \bar y_{3},\, y_{3}\wedge \bar x_{1}\bar y_{1},\, y_{3}\wedge \bar x_{1}\wedge \bar y_{3},\,y_{3}\wedge \bar y_{1}\wedge \bar y_{2},\, y_{3}\wedge \bar y_{2}\wedge \bar y_{3}
\rangle.
\end{multline*}

We compute
\[H^{1,0}_{\bar\partial}(G/\Gamma)=\langle [x_{1}]\rangle,\, H^{0,1}_{\bar\partial}(G/\Gamma)=\langle [\bar x_{1}],\, [\bar y_{2}]\rangle,\]
\[H^{2,0}_{\bar\partial}=0,\, H^{1,1}_{\bar\partial}=\langle [x_{1}\wedge \bar y_{2}], \, [y_{1}\wedge \bar y_{3}],\, [y_{3}\wedge \bar y_{1}]\rangle,\, H^{0,2}_{\bar\partial}=\langle [\bar y_{1}\wedge \bar y_{3}]\rangle,\]
\[H^{3,0}_{\bar\partial}=\langle [y_{1}\wedge y_{2}\wedge y_{3}]\rangle,\, H^{0,3}_{\bar\partial} =\langle[\bar x_{1}\wedge\bar y_{1}\wedge\bar y_{3}],\, [\bar y_{1}\wedge\bar y_{2}\wedge\bar y_{3}]\rangle,\]
\[H^{2,1}_{\bar\partial}=\langle [x_{1}\wedge y_{1}\wedge \bar y_{3}],\, [x_{1}\wedge y_{3}\wedge \bar y_{1}],\,[y_{1}\wedge y_{2}\wedge \bar y_{3}],\, [y_{2}\wedge y_{3}\wedge \bar y_{1}]\rangle,\]
\[H^{1,2}_{\bar\partial}=\langle [x_{1}\wedge \bar x_{1}\wedge \bar y_{2}],\,[x_{1}\wedge \bar y_{1}\wedge \bar y_{3}],\, [y_{1}\wedge \bar x_{1}\wedge  \bar y_{3}],\, [y_{1}\wedge \bar y_{2}\wedge \bar y_{3}],\,  [y_{3}\wedge \bar x_{1}\wedge \bar y_{1}],\, [y_{3}\wedge \bar y_{1}\wedge \bar y_{2}\rangle.
\] 

We also compute
\[H^{1}(G/\Gamma)=\langle[x_{1}],\, [\bar x_{1}], \, [y_{2}+\bar y_{2}]\rangle,
\]
\[H^{2}(G/\Gamma)=\langle [x_{1}\wedge \bar x_{1}],\, [y_{1}\wedge \bar y_{3}],\, [y_{2}\wedge \bar y_{2}-\bar y_{1}\wedge \bar y_{3}+y_{1}\wedge y_{3}],\, [y_{3}\wedge \bar y_{1}]\rangle,
\]
\begin{multline*}
H^{3}(G/\Gamma)=\langle [y_{1}\wedge y_{2}\wedge y_{3}],\, [x_{1}\wedge y_{1}\wedge \bar y_{3}],\, [x_{1}\wedge y_{3}\wedge \bar y_{1}],\,[y_{1}\wedge y_{2}\wedge \bar y_{3}],\, [y_{2}\wedge y_{3}\wedge \bar y_{1}],\\ [y_{3}\wedge\bar x_{1}\wedge \bar y_{1}],\, [-x_{1}\wedge y_{2}\wedge \bar x_{1}+x_{1}\wedge\bar x_{1}\wedge \bar y_{2}],\, [y_{1}\wedge \bar x_{1}\wedge \bar y_{3}],\\ [x_{1}\wedge y_{2}\wedge \bar y_{2}-x_{1}\wedge \bar y_{1}\wedge \bar y_{3}+x_{1}\wedge y_{1}\wedge y_{3}],\, [y_{1}\wedge \bar y_{2}\wedge \bar y_{3}]\\
[-y_{2}\wedge\bar x_{1}\wedge \bar y_{2}-\bar x_{1}\wedge \bar y_{1}\wedge \bar y_{1}+y_{1}\wedge  y_{3}\wedge \bar x_{1}],\,[y_{3}\wedge\bar y_{1}\wedge \bar y_{2}],\, [\bar y_{1}\wedge \bar y_{2}\wedge \bar y_{3}]
\rangle.
\end{multline*}
By these computations,  we have
\[\dim H^{1,0}_{\bar\partial}(G/\Gamma)+\dim H^{0,1}_{\bar\partial}(G/\Gamma)=3=\dim H^{1}(G/\Gamma),
\]
\[\dim H^{2,0}_{\bar\partial}(G/\Gamma)+\dim H^{1,1}_{\bar\partial}(G/\Gamma)+\dim H^{0,2}_{\bar\partial}(G/\Gamma)=4=\dim H^{2}(G/\Gamma),\]
\[\dim H^{3,0}_{\bar\partial}(G/\Gamma)+\dim H^{2,1}_{\bar\partial}(G/\Gamma)+\dim H^{1,2}_{\bar\partial}(G/\Gamma)+\dim H^{0,3}_{\bar\partial}(G/\Gamma)=13=\dim H^{3}(G/\Gamma).
\]
Hence the proposition follows.
\end{proof}
\end{example}

\ \\
{\bf  Acknowledgements.} 
The author would like to express many thanks to Daniele Angella for his  remarks which lead to  improvements in the revised version.


\begin{thebibliography}{40}


\bibitem{CFG} L. A. Cordero, M.  Fern\'andez, A. Gray, The Fr\"olicher spectral sequence for compact nilmanifolds. Illinois J. Math. {\bf 35} (1991), no. 1, 56--67.

\bibitem{CFGU} L. A. Cordero, M. Fern\'andez, A. Gray,  L. Ugarte, A general description of the terms in the Fr\"olicher spectral sequence, Differential Geom. Appl. {\bf 7} (1997), 75--84.
\bibitem{dek} K.  Dekimpe,
Semi-simple splittings for solvable Lie groups and polynomial structures. Forum Math. {\bf 12} (2000), no. 1, 77--96.
\bibitem{Hir} F. Hirzebruch, Topological Methods in Algebraic Geometry, third enlarged ed.,  Springer-Verlag, 1966.
\bibitem{Kd}
H. Kasuya, Techniques of computations of Dolbeault cohomology of  solvmanifolds.   Math. Z. {\bf 273} (2013), no. 1-2, 437--447.
\bibitem{KH} H. Kasuya, Hodge symmetry and decomposition on non-K\"ahler solvmanifolds. 
http://arxiv.org/abs/1109.5929
\bibitem{KDD}
H. Kasuya, de Rham and Dolbeault Cohomology of solvmanifolds with local systems.  arXiv:1207.3988v3
\bibitem{ma} S. Mac Lane, Homology, Springer, 1963.
\bibitem{mc}
J. McCleary, {\em A user's guide to spectral sequences}, Second edition, Cambridge Studies in Advanced Mathematics, \textbf{58}, Cambridge University Press, Cambridge, 2001.
\bibitem{Rol}
S. Rollenske,  The Fr\"olicher spectral sequence can be arbitrarily non-degenerate. Math. Ann. {\bf 341} (2008), no. 3, 623--628.
\bibitem{Vo} C. Voisin, Hodge Theory and complex algebraic geometry I, Cambridge studies in advanced
mathematics, 76, Cambridge University Press 2002.
\end{thebibliography}
\end{document}